\documentclass[11pt, a4paper]{amsart}
\usepackage[T1]{fontenc}
\usepackage[utf8]{inputenc}
\usepackage{geometry}                
\usepackage{graphicx}
\usepackage{amsmath, amsthm, amssymb, amscd}
\usepackage{mathrsfs}
\usepackage{hyperref}
\usepackage[all]{xy}
\usepackage{tikz}
\usepackage{stmaryrd}

\theoremstyle{plain}
\newtheorem{thm}{Theorem}[section]
\newtheorem{lem}[thm]{Lemma}
\newtheorem{prop}[thm]{Proposition}
\newtheorem{cor}[thm]{Corollary}

\newtheorem{thm-defn}[thm]{Theorem-Definition}

\theoremstyle{definition}
\newtheorem{defn}{Definition}[section]

\theoremstyle{remark}
\newtheorem{rem}{Remark}[section]

\newcommand{\gl}{\mathrm{GL}}
\newcommand{\SL}{\mathrm{SL}}
\newcommand{\Sp}{\mathrm{Sp}}
\newcommand{\Gsp}{\mathrm{GSp}}
\newcommand{\SO}{\mathrm{SO}}
\newcommand{\GO}{\mathrm{GO}}

\newcommand{\ks}{\mathfrak{S}}
\newcommand{\pgl}{\mathrm{PGL}}
\newcommand{\bbg}{\mathbb{G}}

\newcommand{\cf}{\mathcal{F}}
\newcommand{\co}{\mathcal{O}}

\newcommand{\cl}{\mathcal{L}}

\newcommand{\cp}{\mathcal{P}}
\newcommand{\cvv}{\mathcal{V}}
\newcommand{\cee}{\mathcal{E}}

\newcommand{\xx}{\mathscr{X}}

\newcommand{\ka}{\mathfrak{a}}
\newcommand{\kg}{\mathfrak{g}}
\newcommand{\kn}{\mathfrak{n}}
\newcommand{\kp}{\mathfrak{p}}
\newcommand{\kt}{\mathfrak{t}}

\newcommand{\bn}{\mathbf{N}}
\newcommand{\bz}{\mathbf{Z}}
\newcommand{\br}{\mathbf{R}}
\newcommand{\bp}{\mathbf{P}}
\newcommand{\bq}{\mathbf{Q}}

\newcommand{\val}{\mathrm{val}}
\newcommand{\lie}{\mathrm{Lie}}
\newcommand{\Ad}{\mathrm{Ad}}

\newcommand{\codim}{\mathrm{Codim}}

\newcommand{\sch}{\mathrm{Sch}}
\newcommand{\Hom}{\mathrm{Hom}}
\newcommand{\ec}{\mathrm{Ec}}
\newcommand{\ind}{\mathrm{ind}}
\newcommand{\gr}{\mathrm{Gr}}
\newcommand{\End}{\mathrm{End}}
\newcommand{\aut}{\mathrm{Aut}}

\newcommand{\rk}{\mathrm{rk}}
\newcommand{\pic}{\mathrm{Pic}}

\newcommand{\lgghat}{\widehat{G(\!(\ep)\!)}}

\newcommand{\rl}{\mathrm{L}}

\newcommand{\ep}{\epsilon}

\newcommand{\xxs}{\mathscr{X}^{\xi}}

\author{\rm Zongbin Chen}

\address{\'Ecole polytechnique fédérale de Lausanne\\
SB, Mathgeom/Geom, MA B1 447, station 8\\
Lausanne, CH-1015 \\
Switzerland.}

\email{zongbin.chen@gmail.com}

\title{The $\xi$-stability on the affine grassmannian}

\begin{document}

\maketitle


\begin{abstract}

We introduce a notion of $\xi$-stability on the affine grassmannian $\xx$ for the classical groups, this is the local version of the $\xi$-stability on the moduli space of Higgs bundles on a curve introduced by Chaudouard and Laumon. We prove that the quotient $\xx^{\xi}/T$ of the stable part $\xx^{\xi}$ by the maximal torus $T$ exists as an ind-$k$-scheme, and we introduce a reduction process analogous to the Harder-Narasimhan reduction for vector bundles over an algebraic curve. 
For the group $\mathrm{SL}_{d}$, we calculate the Poincaré series of the quotient $\xx^{\xi}/T$.

\end{abstract}

\section{Introduction}

Let $k$ be an algebraically closed field, $F=k(\!(\ep)\!)$ the field of Laurent series with coefficients in $k$, $\co=k[\![\ep]\!]$ the ring of integers of $F$, $\kp=\ep k[\![\ep]\!]$ the maximal ideal of $\co$. Let $\val:F^{\times}\to \bz$ be the discrete valuation normalized by $\val(\ep)=1$. 

Let $G$ be a split reductive group over $k$. We have the loop group $\rl G$ (resp. positive loop group $\rl^{+}G$), which is the ind-$k$-scheme representing the functor
$$
\rl G(R)=G(R (\!(t)\!)) \;(\text{resp. }G(R [\![t]\!])), \quad \forall \, k\text{-algebra }R.
$$

The affine grassmannian $\xx^{G}$, by definition, is the quotient $\rl G/ \rl^{+} G$, it has the structure of an ind-$k$-scheme. We will simplify $\xx^{G}$ to $\xx$ when the context is clear. All through the paper, we work with the reduced scheme structure on the affine grassmannian, it is hence enough to work at the level of $k$-rational points $\xx(k)=G(F)/K$, where we denote by $K$ for the maximal compact open subgroup $G(\co)$ of $G(F)$.

Let $T$ be a maximal torus of $G$. We introduce a notion of $\xi$-stability on the affine grassmannian $\xx$, which is a local version of the $\xi$-stability on the moduli space of Higgs bundles introduced by Chaudouard and Laumon \cite{cl}. When $G$ is adjoint, the construction goes roughly as follows: Let $\cp(T)$ be the set of Borel subgroups of $G$ containing $T$. For $x\in \xx,\, B\in \cp(T)$, let $f_{B}(x)\in X_{*}(T)$ be the unique cocharacter $\nu$ such that $x\in U_{B}(F)\ep^{\nu}K/K$, where $U_{B}$ is the unipotent radical of $B$. We identify $X_{*}(T)\otimes \br$ with $\kt$ and we denote $\ec(x)$ for the convex hull of $(f_{B}(x))_{B\in \cp(T)}$ in $\kt$. Given a generic element $\xi\in \kt$, the point $x$ is said to be $\xi$-stable if $\xi\in \ec(x)$.

This notion, a priori combinatorial, is in fact algebro-geometric. Recall that $\xx$ is an ind-$k$-scheme, i.e. there exist projective algebraic varieties $\xx_{n}$ over $k$, together with closed immersions $\xx_{n}\hookrightarrow \xx_{n+1}$, such that  
$$
\xx=\lim_{n\to +\infty}\xx_{n}.
$$
For a particular choice of the injective system $(\xx_{n})_{n\in \bn}$, we show that the $\xi$-stability on $\xx_{n}$ is the same as the stability (in the sense of GIT) under the action of a torus $S_{n}\subset T$ with respect to the Pl\"ucker embedding. ($\xx_{n}$ is a closed subvariety of some Grassmannian, so there is a natural Pl\"ucker embedding.) As a consequence, the quotient $\xx^{\xi}/T$ of the stable part $\xx^{\xi}$ by the torus $T$ exists as an ind-$k$-scheme.

\begin{thm}

Let $G$ be a classical group over $k$, let $T$ be a maximal torus of $G$. The geometric quotient $\xx^{\xi}/T$ exists as an ind-$k$-scheme, and it satisfies the valuative criterion of properness. 

\end{thm}

This construction can also be rephrased by using the action of $T$ but with a twisted polarisation on $\xx_{n}$. It seems to us that this construction can not be put together, i.e. there is no line bundle on $\xx$ such that the $\xi$-stability is the same as the stability under the action of $T$ with respect to this line bundle. In fact, the Picard group of $\xx$ is essentially generated by the determinant line bundle and a simple calculation shows that the stability in this setting coincides with the $0$-stability condition.

Nevertheless, we can introduce a reduction process which is analogous to the Harder-Narasimhan reduction for vector bundles on an algebraic curve. We call it \emph{Arthur-Kottwitz reduction}. It allows us to reduce the non-$\xi$-stable parts onto the $\xi^{M}$-stable parts on the affine grassmannian $\xx^{M}$ associated to the Levi subgroups $M$ of $G$ containing $T$.

In general, the quotient $\xx^{\xi}/T$ is more difficult to study than $\xx$ itself. For the group $\SL_{d}$, we are able to calculate the Poincaré series of $\xx^{\xi}/T$. 

\begin{thm}

For $G=\mathrm{SL}_{d}$, the Poincaré series of $\xx^{\xi}/T$ is
$$
\frac{1}{(1-t^{2})^{d-1}}\prod_{i=1}^{d-1}(1-t^{2i})^{-1}.
$$ 
\end{thm}

There are two ingredients in the proof: the first one is the fact that $\xx^{\xi}/T$ is homologically smooth and hence satisfies the Poincaré duality, the second one is to give a lower bound on the codimension of the non-$\xi$-stable parts.

Our motivation to introduce the $\xi$-stability on the affine grassmannian comes from our study on the affine Springer fibers. By definition, the affine Springer fiber associated to an integral regular semi-simple element $\gamma\in \kg(F)$ is 
$$
\xx_{\gamma}=\{g\in G(F)/K\mid \Ad(g^{-1})\gamma\in \kg(\co)\}.
$$
It is the local analogue of the fibers of the Hitchin fibration.
We refer the reader to \cite{go} for a summary of the basic properties of the affine Springer fibers. It is conjectured by Goresky, Kottwitz and MacPherson \cite{gkm} that the affine Springer fibers are cohomologically pure in the sense of Deligne. One of the difficulties to attack this conjecture comes from the fact that generally the affine Springer fibers are not of finite type. For example, let $\gamma\in \kt(\co)$, then $T(F)$ acts on $\xx_{\gamma}$ and this induces an isomorphism between $T(F)/T(\co)$ and the irreducible components of $\xx_{\gamma}$. This difficulty can be overcome if we consider $\xx^{\xi}_{\gamma}/T$. It is reasonable to expect that the quotient $\xx_{\gamma}^{\xi}/T$ is still cohomologically pure and that the purity of $\xx_{\gamma}$ is equivalent to that of $\xx_{\gamma}^{\xi}/T$. This gives a possibility to attack the purity conjecture. In fact,  
Chaudouard has done some calculations for the group $G=\mathrm{SL}_{2}$ and his calculation is the starting point of the story.

\subsection*{Notations}

Let $\Phi=\Phi(G,T)$ be the root system of $G$ with respect to $T$, let $W$ be the Weyl group of $G$ with respect to $T$. For any subgroup $H$ of $G$ which is stable under the conjugation of $T$, we note $\Phi(H,T)$ for the roots appearing in $\lie(H)$. We use the $(G,M)$ notation of Arthur \cite{a}. Let $\cf(T)$ be the set of parabolic subgroups of $G$ containing $T$, let $\cl(T)$ be the set of Levi subgroups of $G$ containing $T$. For every $M\in \cl(T)$, we denote by $\cp(M)$ the set of parabolic subgroups of $G$ whose Levi factor is $M$. Let $X^*(M)=\Hom(M, \bbg_m)$ and $\ka_M^{*}=X^*(M)\otimes\br$. The restriction $X^{*}(M)\to X^{*}(T)$ induces an injection $\ka_{M}^{*}\hookrightarrow \ka_{T}^{*}$. Let $(\ka_{T}^{M})^{*}$ be the subspace of $\ka_{T}^{*}$ generated by $\Phi(M,T)$. We have the decomposition in direct sums
$$
\ka_{T}^{*}=(\ka_{T}^{M})^{*}\oplus \ka_{M}^{*}.
$$

The canonical pairing 
$$
X_{*}(T)\times X^{*}(T)\to \bz
$$ 
can be extended linearly to $\ka_{T}\times \ka_{T}^{*}\to \br$, with $\ka_{T}=X_{*}(T)\otimes \br$. For $M\in \cl(T)$, let $\ka_{T}^{M}\subset \ka_{T}$ be the subspace orthogonal to $\ka_{M}^{*}$, and $\ka_{M} \subset \ka_{T}$ be the subspace orthogonal to $(\ka_{T}^{M})^{*}$, then we have the decomposition
$$
\ka_{T}=\ka_{M}\oplus \ka_{T}^{M},
$$
let $\pi_{M},\,\pi^{M}$ be the projections to the two factors.

\subsection*{Acknowledgement} 

This article is based on part of the author's thesis at Université Paris-Sud 11 at Orsay. We want to thank G\'erard Laumon for having posed this question, and for his encouragements during the preparation of this work. We also want to thank an anonymous referee for his careful reeding and helpful suggestions.

\section{Recall on affine grassmannians}

\subsection{Affine grassmannians via lattices}

The general references for this section are \cite{bl}, \cite{go}, \cite{faltings} and \cite{pr}. 

For any $k$-algebra $R$, a \emph{lattice} $L$ of $R(\!(\ep)\!)^{d}$ is a $R[\![\ep]\!]$-submodule of $R(\!(\ep)\!)^{d}$ satisfying
\begin{enumerate}
\item
there exists $N\in \bn$ such that
$$
\ep^{N}R[\![\ep]\!]^{d}\subset L\subset \ep^{-N}R[\![\ep]\!]^{d},
$$

\item

the quotient $L/\ep^{N}R[\![\ep]\!]^{d}$ is a projective $R$-module.

\end{enumerate}

We denote by $L_{0}$ the standard lattice $R[\![\ep]\!]^{d}$. The \emph{index} of a lattice $L$ is defined to be
$$
\ind(L)=\rk_{R}(L_{0}/\ep^{N}R[\![\ep]\!]^{d})-\rk_{R}(L/\ep^{N}R[\![\ep]\!]^{d}),  
$$
for any sufficiently large $N$.

One can prove that the affine grassmannian $\xx^{\gl_{d}}$ (resp. $\xx^{\SL_{d}}$) represents the functor which associates to every $k$-algebra $R$ the set of lattices in $R(\!(\ep)\!)^{d}$ (resp. lattices of index $0$ in $R(\!(\ep)\!)^{d}$). Further more, $\xx^{\SL_{d}}$ is naturally identified with the central connected component of $\xx^{\gl_{d}}$, and every connected component of $\xx^{\gl_{d}}$ is isomorphic to $\xx^{\SL_{d}}$. It is then enough to understand the ind-$k$-scheme structure on $\xx^{\SL_{d}}$.

For each $n\in \bn$, let 
$$
\xx_{n}^{\SL_{d}}=\{L\in \xx^{\SL_{d}}\mid L\subset \ep^{-n}L_{0} \}.
$$
We can embed $\xx_{n}^{\SL_{d}}$ as a closed sub-variety of some grassmannian.

\begin{lem}\label{embedA}

The algebraic variety $\xx_{n}^{\SL_{d}}$ is a Springer fiber.

\end{lem}

\begin{proof}

For $L\in \xx_{n}$, we have automatically $\ep^{(d-1)n}L_{0}\subset L$. Let $\gr_{nd(d-1), nd^{2}}$ be the grassmannian of sub vector spaces of dimension $nd(d-1)$ in $k^{nd^{2}}$. We embed $\xx_{n}$ in $\gr_{nd(d-1), nd^{2}}$ by the injective morphism $\varrho_{n}:\xx_{n}\to \gr_{nd(d-1), nd^{2}}$ defined by
$$
\varrho_{n}(L)=L/\ep^{n(d-1)}L_{0}\subset \ep^{-n}L_{0}/\ep^{n(d-1)}L_{0}.
$$
The image of $\varrho_{n}$ is the Springer fiber 
$$
Y_{n}=\{V\in \gr_{nd(d-1), nd^{2}}\mid NV\subset V\},
$$
where $N\in \End(\ep^{-n}L_{0}/\ep^{n(d-1)}L_{0})$ is the endomorphism defined by the multiplication by $\ep$.

\end{proof}

It is clear that there is a natural closed embedding $\iota_{n}: \xx^{\SL_{d}}_{n}\to \xx^{\SL_{d}}_{n+1}$ for all $n$, and  
$$
\xx^{\SL_{d}}=\lim_{n\to +\infty} \xx^{\SL_{d}}_{n},
$$
from which we get the ind-$k$-scheme structure on $\xx^{\SL_{d}}$. The above construction also gives us a natural line bundle $\cl$ on the affine grassmannian: For $x\in \gr_{nd(d-1), nd^{2}}$, let $E_{x}\subset k^{nd^{2}}$ be the sub vector space corresponding to $x$, we have the tautological vector bundle $\cee_{n}$ on $\gr_{nd(d-1), nd^{2}}$ defined by
$$
\cee_{n}=\{(x,v)\in \gr_{nd(d-1), nd^{2}}\times k^{nd^{2}}\mid v \in E_{x}\}.
$$
The line bundle $\bigwedge^{nd(d-1)} \cee_{n}$ is the familiar one that was used to define the Pl\"ucker embedding of $\gr_{nd(d-1), nd^{2}}$. Let $\cl_{n}= \varrho_{n}^{*} \bigwedge^{nd(d-1)} \cee_{n}$, we get a line bundle on $\xx^{\SL_{d}}_{n}$. It is easy to verify that $\iota_{n}^{*}\cl_{n+1}=\cl_{n}$, so 
$$
\cl=\lim_{n\to +\infty}\cl_{n}
$$
defines a line bundle on $\xx^{\SL_{d}}$. A similar construction works on $\xx^{\gl_{d}}$. We call them the \emph{determinant line bundle} on the affine grassmannian. It is known that $\pic(\xx^{\SL_{d}})$ is freely generated by the determinant line bundle.

With the determinant line bundle, we can put the Pl\"ucker embedding for all the $\xx^{\SL_{d}}_{n}$ into one infinite dimensional Pl\"ucker embedding, they are just the inductive limit of Pl\"ucker embedding at each level. Let $V=H^{0}(\xx^{\SL_{d}}, \cl)$, let 
$
\rho: \xx^{\SL_{d}} \to \bp(V)
$
be the resulting infinite dimensional Pl\"ucker embedding. It has another representation theoretic construction, which will be explained in the next section.

For general connected reductive group $G$ over $k$, as explained in \cite{go}, we can choose a closed embedding $G\to \gl_{N}$ for some large enough $N$, and use the following lemma of Beilinson and Drinfeld, since $\gl_{N}/G$ is affine.

\begin{lem}[\cite{bd}, proof of theorem 4.5.1.]\label{BDcriteria}

Let $G_{1}\subset G_{2}$ be linear algebraic groups over $k$ such that the quotient $U:=G_{2}/G_{1}$ is quasi-affine. Suppose that the quotient $\rl G_{2}/\rl^{+} G_{2}$ is an ind-$k$-scheme of ind-finite type. Then the same holds for $\rl G_{1}/\rl^{+} G_{1}$, and the natural morphism $\rl G_{1}/\rl^{+} G_{1}\to \rl G_{2}/\rl^{+} G_{2}$ is a locally closed immersion. If $U$ is affine, then this immersion is a closed immersion.

\end{lem}

The connected components of the affine grassmannian can be described as follows. We identify $X_{*}(T)$ with $T(F)/T(\co)$ by sending $\chi$ to $\ep^{\chi}:=\chi(\ep)$. With this identification, the canonical surjection $T(F)\to T(F)/T(\co)$ can be viewed as
\begin{equation}\label{indexT}
T(F)\to X_{*}(T).
\end{equation}

Let $\Lambda_{G}$ be the quotient of $X_{*}(T)$ by the coroot lattice of $G$ (the subgroup of $X_{*}(T)$ generated by the coroots of $T$ in $G$). We have a canonical homomorphism
\begin{equation}\label{indexM}
G(F)\to \Lambda_{G},
\end{equation}
which is characterized by the following properties: it is trivial on the image of $G_{\mathrm{sc}}(F)$ in $G(F)$ ($G_{\mathrm{sc}}$ is the simply connected cover of the derived group of $G$), and its restriction to $T(F)$ coincides with the composition of (\ref{indexT}) with the projection of $X_{*}(T)$ to $\Lambda_{G}$. Since the morphism (\ref{indexM}) is trivial on $G(\co)$, it descends to the \emph{Kottwitz map}
$$
\kappa_{G}:\xx^{G}\to \Lambda_{G},
$$
whose fibers are the connected components of $\xx^{G}$. (For the group $\gl_{d}$, the Kottwitz map sends a lattice $L$ in $F^{d}$ to its index.) Furthermore, all the connected components are isomorphic to the central one $\xx^{G_{\mathrm{sc}}}$. To see this, for any element $\lambda\in \Lambda_{G}$, let $\tilde{\lambda}\in X_{*}(T)$ be any lifting of $\lambda$, then the translation by $\ep^{\tilde{\lambda}}$ induces an isomorphism between the central connected component and $\kappa_{G}^{-1}(\lambda)$. It is also clear that this isomorphism commutes with the $T$-action.
So we can restrict to the simply connected simple algebraic groups without loss of generality.

\subsection{Affine grassmannian via representation theory}

The basic references for this section are \cite{ps}, \cite{ku}. (Both of them work in the characteristic $0$ setting, but since the Kac-Moody group schemes over $\bz$ have been constructed by Tits \cite{t1}, \cite{t2}, what we recount below should extend to any field.)

Let $G$ be a simply connected simple algebraic group over $k$ of rank $r$. We have the loop group $G(\!(\ep)\!)$,
let $\lgghat$ be the universal central extension of $G(\!(\ep)\!)\rtimes \bbg_{m}$, where $\bbg_{m}$ is the rotation torus, it acts trivially on $G$ and acts by $t*\ep^{n}=t^{n}\ep^{n}$ on the loops.

Let $\widehat{G[\![\ep]\!]}$ be the pre-image of $G[\![\ep]\!]$ under the natural projection $\lgghat\to G(\!(\ep)\!)$. It is known that $\widehat{G[\![\ep]\!]}=G[\![\ep]\!]\times \bbg_{m}$. Let $\delta$ be the character of $G[\![\ep]\!]\times \bbg_{m}$ which is trivial on $G[\![\ep]\!]$ while tautological on $\bbg_{m}$. Let $\cl=\lgghat\times_{(G[\![\ep]\!]\times \bbg_{m})} k$, where $G[\![\ep]\!]\times \bbg_{m}$ acts on $k$ via the character $\delta$. It defines a line bundle on $\lgghat/(G[\![\ep]\!]\times \bbg_{m})=\xx^{G}$. We call it the \emph{determinant line bundle} since it generalises the determinant line bundle on $\xx^{\SL_{d}}$. It is known that $\pic(\xx^{G})$ is freely generated by $\cl$.

The global section $V=H^{0}(\xx^{G}, \cl)$ is a realisation of the integrable highest weight irreducible representation of $\lgghat$ of highest weight $\delta$. Let $v_{0}$ be the highest weight vector, we have the infinite dimensional Pl\"ucker embedding $\rho: \xx^{G}=\lgghat/\widehat{G[\![\ep]\!]}\to \bp(V)$ sending $[g]$ to the line generated by $gv_{0}$. For the group $\SL_{d}$, it coincides with the infinite dimensional Pl\"ucker embedding constructed in the previous section.

The torus $T$ acts on $\xx^{G}$ by left translation, with fixed points $X_{*}(T)$. This action lifts to $\cl$ since $T\subset \lgghat$. In particular, it acts on the fibers $\cl_{\chi},\,\chi\in X_{*}(T)$ with a certain character $\Theta_{\chi}\in X^{*}(T)$. This has been calculated by Mirkovic and Vilonen \cite{mv}. They show that the map $X_{*}(T)\to X^{*}(T)$ sending $\chi $ to $\Theta_{\chi} $ is linear, and its restriction to $\alpha^{\vee},\,\alpha\in \Phi(G, T),$ is 
$$
\Theta_{\alpha^{\vee}}=\frac{2}{(\alpha,\alpha)^{*}}\alpha,
$$ 
where $(\,,)^{*}$ is the invariant non-degenerate symmetric bilinear form on $\kt^{*}$ normalised by $(\theta,\theta)^{*}=2$ for the longest root $\theta$.

\begin{rem}
Observe that $T$ acts on the fiber $\cl_{0}$ of $\cl$ at the point $\ep^{0}$ trivially. This is not in contradiction with the classical construction of the determinant line bundle for $\SL_{d}$. In fact, in the infinite dimensional situation, the ``absolute'' determinant of $T$ action doesn't make sense, it is the ``relative'' determinant which is meaningful, here we use its action on $\cl_{0}$ as the reference.

\end{rem}

\section{Existence of $\xx^{\xi}/T$ as an ind-$k$-scheme}

\subsection{The notion of $\xi$-stability}

For $M\in \cl(T)$, the natural inclusion $M\hookrightarrow G$ induces a closed immersion of $\xx^{M}$ in $\xx^{G}$ by lemma \ref{BDcriteria}. For $P=MN\in \cf(T)$, we have the retraction
$$
f_{P}:\xx\to \xx^{M}
$$ 
which sends $gK=nmK$ to $mM(\co)$, where $g=nmk,\,n\in N(F),\, m\in M(F),\, k\in K$ is the Iwasawa decomposition. More generally we can define $f^{L}_{P_{L}}:\xx^{L}\to \xx^{M}$ for $L\in \cl(T), \, L\supset M$ and $P_{L}\in \cp^{L}(M)$. These retractions satisfy the transition property: Suppose that $Q\in \cp(L)$ satisfy $Q \supset P$, then
$$
f_{P}=f^{L}_{P\cap L}\circ f_{Q}.
$$

We want to point out that the retraction $f_{P}$ is not an algebraic morphism between ind-$k$-schemes. In fact, it is not even a continuous morphism. But it becomes an algebraic morphism when restricted to the inverse image of each connected component of $\xx^{M}$.

\begin{prop}\label{fibrationp}

For $\lambda\in \Lambda_{M}$, let $\xx^{M,\lambda}=\kappa_{M}^{-1}(\lambda)$. The inverse image $f_{P}^{-1}(\xx^{M,\lambda})=N(F)\xx^{M,\lambda}$ is locally closed in $\xx$, and the retraction $f_{P}$ makes it an infinite dimensional homogeneous affine fibration on $\xx^{M,\lambda}$. 
\end{prop}

\begin{proof}

Up to translation, we can suppose that $\lambda=0$. Let $P_{\mathrm{sc}}=M_{\mathrm{sc}}N$, we have $f_{P}^{-1}(0)=P_{\mathrm{sc}}(F)/P_{\mathrm{sc}}(\co)$. Since $G/P_{\mathrm{sc}}$ is quasi-affine, the inclusion $P_{\mathrm{sc}}(F)/P_{\mathrm{sc}}(\co)\to G(F)/G(\co)$ is a locally closed immersion by lemma \ref{BDcriteria}.

Given any point $mK\in \xx^{M,0}, \,m\in M_{\mathrm{sc}}(F)$, the orbit $N(F)mK/K$ is isomorphic to 
$$
N(F)/[N(F)\cap \Ad(m)K]\cong \kn(F)/[\kn(F)\cap \Ad(m)\kg(\co)].
$$
For any other point $m'mK\in \xx^{M,0}$ with $m'\in M_{\mathrm{sc}}(F)$, the translation by $m'$ induces a canonical isomorphism between the orbits, i.e.
$$
m'N(F)mK/K=(m'N(F)m'^{-1})m'mK/M=N(F)m'mK/K.
$$
This justifies the second assertion.

\end{proof}

We have the function $H_{P}:\xx\to \ka_{M}^{G}:=\ka_{M}/\ka_{G}$ which is the composition of $\kappa_{M}\circ f_{P}$ and the natural projection of $\Lambda_{M}$ to $\ka_{M}^{G}$.

\begin{prop}[Arthur]

Let $B',B''\in \cp(T)$ be two adjacent Borel subgroups, let $\alpha_{B',B''}^{\vee}$ be the coroot which is positive with respect to $B'$ and negative with respect to $B''$. Then for any $x\in \xx$, we have
$$
H_{B'}(x)-H_{B''}(x)=n(x,B',B'')\cdot \alpha_{B',B''}^{\vee},
$$
with $n(x, B', B'')\in \bz_{\geq 0}$.

\end{prop}

\begin{proof}

Let $P$ be the parabolic subgroup generated by $B'$ and $B''$, let $P=MN$ be the Levi factorization. The application $H_{B'}$ factor through $f_{P}$, i.e. we have commutative diagram
$$
\xymatrix{
\xx \ar[d]_{f_{P}} \ar[dr]^{H_{B'}} &\\
\xx^{M}\ar[r]_{H^{M}_{B'\cap M}}&\ka_{T}^{G}}
$$
and similarly for $H_{B''}$. Since $M$ has semisimple rank $1$, the proposition is thus reduced to $G=\mathrm{SL}_{2}$. In this case, let $T$ be the maximal torus of the diagonal matrices, $B'=\begin{pmatrix}*&*\\ & * 
\end{pmatrix}
$,
$
B''=\begin{pmatrix}*&\\ * & * 
\end{pmatrix}$, and we identify $\ka_{T}^{G}$ with the line $H=\{(x,-x)\mid x\in \br\}\subset \br^{2}$ in the usual way. By the Iwasawa decomposition, any point $x\in \xx$ can be written as
$
x=\begin{pmatrix}a&b \\ &d\end{pmatrix}K.
$ 
Let $m=\min\{\val(a),\,\val(b)\}$, $n=\val(d)$, then $m+n\leq \val(a)+\val(d)=0$ and
$$
H_{B'}(x)=(-n,n), \quad H_{B''}(x)=(m,-m).
$$
So 
$$
H_{B'}(x)-H_{B''}(x)=(-(n+m), n+m)=-(n+m)\cdot \alpha_{B',B''}^{\vee},
$$
and the proposition follows.
\end{proof}

\begin{defn}

For any point $x\in \xx$, we denote by $\ec(x)$ the convex hull in $\ka_{T}^{G}$ of the $H_{B'}(x),\,B'\in \cp(T)$.

\end{defn}

\begin{defn}\label{generic}

Let $\xi\in \ka_{T}^{G}$, it is said to be \emph{generic} if $\alpha(\xi)\notin \bz,\,\forall \alpha\in \Hom(X_{*}(T),\,\bz)$.

\end{defn}

In the following, we always suppose that $\xi$ is generic.

\begin{defn}

For any point $x\in \xx$, we say that it is \emph{$\xi$-stable} if $\xi\in \ec(x)$.

\end{defn}

Let $\xx^{\xi}$ be the set of $\xi$-stable points in $\xx$, let $\xx^{\text{n-}\xi}$ be the set of non-$\xi$-stable points, both of them are $T$-invariant. 
From the definition of $\xi$-stability, it is clear that $x\in \xx^{\xi}$ if and only if $\ec(x)$ contains the image in $\ka_{T}^{G}$ of the alcove in which $\xi$ sits in. In particular, $\dim(\ec(x))=\dim(\ka_{T}^{G})$, so $T/Z_{G}$ acts freely on $\xx^{\xi}$, where $Z_{G}$ is the center of $G$. Moreover, given $\xi, \xi'\in \ka_{T}^{G}$, there exists $w$ in the extended affine Weyl group $\widetilde{W}$ such that $w$ translates the alcove containing $\xi$ to that containing $\xi'$. So $\xx^{\xi}$ is independent of the choice of $\xi$.

The ind-$k$-scheme structure of both $\xx^{\xi}$ and $\xx^{\xi}/T$ will be the main concern of the rest of the section. Roughly speaking, we will show that the notion of $\xi$ stability is the same as that of stability in the geometric invariant theory of Mumford. This allows us to conclude that $\xx^{\xi}$ is an open sub-ind-$k$-scheme of $\xx$ and the quotient $\xx^{\xi}/T$ exists as an ind-$k$-scheme.
More precisely, for a particular choice of the injective system $\xx_{n}$ such that $\xx=\varinjlim \xx_{n}$, we prove that the $\xi$-stability on $\xx_{n}$ coincides with the stability in the sense of Mumford under the action of a ``twisted'' torus $S_{n}\subset T$ with respect to the Pl\"ucker embedding. This is the same as the stability under the action of $T$ with respect to a ``twisted'' Pl\"ucker embedding. More precisely, it is the composition of the Pl\"ucker embedding with a further ``twisting'': $\bp^{N}\to \bp^{N}$. We should point out that the torus $S_{n}$ differs for different $n$. Our result is summarised in the following proposition.

\begin{prop}\label{mainquotient}

Let $G$ be a classical group over $k$. The subset $\xx^{\xi}$ is in fact an open sub-ind-$k$-scheme of $\xx$. The geometric quotient $\xx^{\xi}/T$ of $\xx^{\xi}$ by $T$ exists as an ind-$k$-scheme. In fact, it is the direct limit of a family of projective varieties over $k$.

\end{prop}

Although the proof is a case by case analysis, the main idea remains the same. So we will give a detailed proof only for $\gl_{d}$, and indicate the modifications for the other classical groups.

\begin{rem}
One may try to look for a line bundle on $\xx$ such that the $\xi$-stability on $\xx$ coincides with the stability under the action of $T$ with respect to this line bundle. It seems to us that this is impossible. The Picard group of $\xx$ is essentially generated by the determinant line bundle.  
A simple calculation using the result in \S 1.2 on $T$ action on the fiber $\cl_{\chi},\chi\in X_{*}(T)$ shows that, the stability condition in this setting coincides with $0$-stability.

\end{rem}

\subsection{The group $\gl_{d}$}

Let $G=\gl_{d}$, let $T$ be the maximal torus of the diagonal matrices. The affine grassmannian parametrizes the lattices in $F^{d}$, i.e.
$$
\xx=\{L\subset F^{d}\mid L \text{ is an }\co\text{-module of finite type such that }L\cdot F=F^{d}\}.
$$

Recall that we have the map $\ind:\xx\to \bz$ sending a lattice in $F^{d}$ to its index. For $n\in \bz$, let $\xx^{(n)}=\ind^{-1}(n)$. Since all the connected components of $\xx$ are translations of the neutral connected component $\xx^{(0)}$, it is thus enough to study $\xx^{(0)}=\xx^{\SL_{d}}$. Let $\{e_{i}\}_{i=1}^{d}$ be the natural basis of $F^{d}$ over $F$.

\begin{prop}\label{stA}

Let $\xi\in \kt$ be such that $\sum_{i=1}^{d}\xi_{i}=0$. A lattice $L\in \xx$ of index $0$ is $\xi$-stable if and only if for any permutation $\tau\in \ks_{d}$, we have 
$$
\xi_{\tau(1)}+\cdots+ \xi_{\tau(i)}\leq \ind(L\cap (Fe_{\tau(1)}\oplus \cdots \oplus Fe_{\tau(i)})),\quad i=1,\cdots, d.
$$

\end{prop}

\begin{proof}

Let $B'=\tau(B)$, let $H_{B'}(L)=(n_{1},\cdots,n_{d})$, then we have
$$
n_{\tau(1)}+\cdots+n_{\tau(i)}=\ind(L\cap (Fe_{\tau(1)}\oplus \cdots \oplus Fe_{\tau(i)})),
$$
and the proposition follows.
\end{proof}

We use the same notation as in \S 2.1. We will prove the following result, which implies the proposition \ref{mainquotient}.

\begin{prop}\label{stablequotient}

The subset $\xx_{n}^{\xi}$ is an open subvariety of $\xx_{n}$ and the quotient $\xx_{n}^{\xi}/T$ is a projective $k$-variety. 

\end{prop}

\subsubsection{A non-standard quotient of the grassmannian}

Let $E_{1},\cdots,E_{d}$ be vector spaces over $k$ of dimension $\dim(E_{i})=N_{i}$. Let $X$ be the grassmannian of sub vector spaces of dimension $n$ of $E_{1}\oplus \cdots \oplus E_{d}$. We have the Pl\"ucker immersion $\iota: X\to \bp^{N}$, $N=\binom{N_{1}+\cdots+N_{d}}{n}-1$, and the line bundle $\cl=\iota^{*}\co_{\bp^{N}}(1)$ is naturally endowed with a $\aut(E_{1}\oplus\cdots \oplus E_{d})$-linearization.

The torus $T=\bbg_{m}^{d}$ acts on $E_{1}\oplus\cdots \oplus E_{d}$ with its $i$-th factor acts as homothetie on $E_{i}$, thus it acts on $X$. Given $r_{i}, s_{i}\in \bn$, let $S\subset T$ be the sub torus of $T$ defined by

\begin{equation}\label{soustore}
S=\left\{\begin{bmatrix}t_{1}^{r_{1}}&&&\\&\ddots&&\\ &&t_{d-1}^{r_{d-1}}&\\&&&t_{1}^{-s_{1}}\cdots t_{d-1}^{-s_{d-1}}\end{bmatrix};\;t_{i}\in k^{\times}\right\}.
\end{equation}

We will give a geometric description of the (semi-)stable points of $X$ under the action of $S$ with respect to the polarization given by the line bundle $\cl$, using the criteria of Hilbert-Mumford. Let $Z$ be a projective algebraic variety over $k$ endowed with the action of a reductive group $H$ over $k$, let $\cl$ be an ample $H$-equivariant line bundle over $Z$. Let $\lambda:\bbg_{m}\to H$ be a homomorphism of algebraic group, then $\bbg_{m}$ acts on $Z$ via the morphism $\lambda$. For any point $x\in Z$, the point $x_{0}:=\lim_{t\to 0}\lambda(t)x$ exists since $Z$ is projective. Obviously $x_{0}\in Z^{\bbg_{m}}$, thus $\bbg_{m}$ acts on the stalk $\cl_{x_{0}}$. The action is given by a character of $\bbg_{m}$, $\alpha:t\to t^{r}$, for some $r\in \bz$. Let $\mu^{\cl}(x,\lambda)=-r$.

\begin{thm}[Hilbert-Mumford, \cite{m}, p. 49-54]

Let $Z$ be a projective algebraic variety over $k$ endowed with the action of a reductive group $H$ over $k$, let $\cl$ be an ample $H$-equivariant line bundle over $Z$. Let $Z^{ss}$ (resp. $Z^{st}$) be the open subvariety of $Z$ of the semi-stable (resp. stable) points. Then for any geometric point $x\in Z$, we have

\begin{enumerate}

\item $x\in Z^{ss} \iff \mu^{\cl}(x,\lambda)\geq 0, \,\forall \, \lambda\in \Hom(\bbg_{m}, H) $,

\item $x\in Z^{st} \iff \mu^{\cl}(x,\lambda)> 0, \,\forall \, \lambda\in \Hom(\bbg_{m}, H) $.

\end{enumerate}

\end{thm}

\begin{lem}

We have
$$
V\in X^{S}\iff V=V_{1}\oplus\cdots\oplus V_{d},
$$
where $V_{i}\subset E_{i}$ is a sub vector space.
\end{lem}

For $\mathbf{n}=(n_{1},\cdots,n_{d-1})\in \bz^{d-1}$ such that the cocharacter $\lambda_{\mathbf{n}}\in X_{*}(S)$ defined by 
$$
\lambda_{\mathbf{n}}(t)=\begin{bmatrix}t^{n_{1}r_{1}}&&&\\&\ddots&&\\ &&t^{n_{d-1}r_{d-1}}&\\&&&t^{-\sum_{i=1}^{d-1}n_{i}s_{i}}\end{bmatrix}
$$
is regular. The stability condition is equivalent to the condition that $-\mu^{\cl}(V,\lambda_{\mathbf{n}})<0$ for all such $\mathbf{n}\in \bz^{d-1}$. Up to conjugation, we can suppose that
 
\begin{equation}
n_{1}r_{1}<\cdots<n_{i}r_{i}<-\sum_{j=1}^{d-1}n_{j}s_{j}<n_{i+1}r_{i+1}<\cdots<n_{d-1}r_{d-1}.\label{cond1}
\end{equation}

Let
$
\lim_{t\to 0}\lambda_{\mathbf{n}}(t)V=V_{1}\oplus\cdots\oplus V_{d},
$ where $V_{i}\subset E_{i}$ is a sub vector space of dimension $a_{i}$. We have the relation

\begin{eqnarray}\label{d}
a_{1}&=&\dim(V)-\dim(V\cap (E_{2}\oplus\cdots\oplus E_{d})),\\&\vdots&\nonumber\\ a_{i}&=&\dim(V\cap (E_{i}\oplus\cdots\oplus E_{d}))-\dim(V\cap (E_{i+1}\oplus\cdots\oplus E_{d})),\nonumber \\a_{d}&=&\dim(V\cap (E_{i+1}\oplus\cdots\oplus E_{d}))-\dim(V\cap (E_{i+1}\oplus\cdots\oplus E_{d-1})),\nonumber\\ a_{i+1}&=&\dim(V\cap (E_{i+1}\oplus\cdots\oplus E_{d-1}))-\dim(V\cap (E_{i+2}\oplus\cdots\oplus E_{d-1}))\nonumber \\&\vdots&\nonumber\\
a_{d-1}&=&\dim(V\cap E_{d-1}).\nonumber
\end{eqnarray}

So the stability condition can be written as

\begin{eqnarray}-\mu^{\cl}(V,\lambda_{\mathbf{n}})&=&\sum_{i=1}^{d-1}a_{i}n_{i}r_{i}-a_{d}\sum_{j=1}^{d-1}n_{j}s_{j}\nonumber \\
&=&\sum_{i=2}^{d-1}a_{i}(n_{i}r_{i}-n_{1}r_{1})-a_{d}\left(n_{1}r_{1}+\sum_{j=1}^{d-1}n_{i}s_{i}\right)+nn_{1}r_{1}<0,\label{cond2}
\end{eqnarray}
(We use the relation $n=a_{1}+\cdots+a_{d}$ in the second equality.)

The equality (\ref{cond2}) is a question of maximal value of a linear functional on a convex region, so it suffices to look at the condition at the boundary, i.e.

\begin{eqnarray}\label{cond3}
n_{1}r_{1}=\cdots=n_{i_{0}}r_{i_{0}}<n_{i_{0}+1}r_{i_{0}+1} & =&\cdots  = -\sum_{j=1}^{d-1}n_{j}s_{j}\\ \nonumber &=&\cdots
 =n_{d-1}r_{d-1}, \quad  1\leq i_{0}\leq i;\end{eqnarray}

and

\begin{eqnarray}\label{cond4}
n_{1}r_{1}=\cdots=-\sum_{j=1}^{d-1}n_{j}s_{j}=\cdots &=& n_{i_{0}}r_{i_{0}}<n_{i_{0}+1}r_{i_{0}+1}\\
\nonumber &=& \cdots =n_{d-1}r_{d-1},
\quad  i+1\leq i_{0}\leq d-1.\end{eqnarray}

The inequality (\ref{cond3}) gives

$$
n_{1}<0,\quad n_{i_{0}+1}r_{i_{0}+1}(1+\sum_{j=i_{0}+1}^{d-1}s_{j}/r_{j})=-n_{1}r_{1}\sum_{j=1}^{i_{0}}s_{j}/r_{j},
$$

so the inequality (\ref{cond2}) implies

\begin{equation}\label{s1}
a_{i_{0}+1}+\cdots+a_{d}<\frac{n(1+\sum_{j=i_{0}+1}^{d-1}s_{j}/r_{j})}{1+\sum_{j=1}^{d-1}s_{j}/r_{j}},\quad 1\leq i_{0}\leq i.
\end{equation}

The inequality (\ref{cond4}) gives
$$
n_{1}<0,\quad n_{i_{0}+1}r_{i_{0}+1}\sum_{j=i_{0}+1}^{d-1}s_{j}/r_{j}=-n_{1}r_{1}(1+\sum_{j=1}^{i_{0}}s_{j}/r_{j}),
$$

so the inequality (\ref{cond2}) implies

\begin{equation}\label{s2}
a_{i_{0}+1}+\cdots+a_{d-1}<\frac{n(\sum_{j=i_{0}+1}^{d-1}s_{j}/r_{j})}{1+\sum_{j=1}^{d-1}s_{j}/r_{j}},\quad i+1\leq i_{0}\leq d-1.
\end{equation}

We can express the inequalities (\ref{s1}) and (\ref{s2}) as inequalities in $\dim(V\cap (E_{i_{1}}\oplus\cdots\oplus E_{i_{r}}))$ with the help of the dimension relation (\ref{d}).

Let $x=(x_{1},\cdots,x_{d})\in \kt$ with 
$$
x_{d}=\frac{1}{1+\sum_{i=1}^{d-1}s_{i}/r_{i}},\quad x_{i}=\frac{s_{i}/r_{i}}{1+\sum_{i=1}^{d-1}s_{i}/r_{i}},\; i=1,\cdots,d-1.
$$ 
We remark that $\sum_{i=1}^{d}x_{i}=1$. The above calculations, after suitable permutation, can be reformulated as:

\begin{prop}\label{st1}

A sub vector space $V\subset E_{1}\oplus\cdots\oplus E_{d}$ is $S$-stable if and only if for any permutation $\tau\in \mathfrak{S}_{d}$, we have
$$
\dim(V)(x_{\tau(1)}+\cdots+x_{\tau(i)})> \dim(V\cap(E_{\tau(1)}\oplus\cdots\oplus E_{\tau(i)})),\quad i=1,\cdots,d-1.
$$
\end{prop}

The same result holds for the semi-stable points with ``$>$'' replaced by ``$\geq$''.

\subsubsection{Comparison of two notions of stability}

Recall that in \S 2.1 we have constructed a closed immersion $\varrho_{n}:\xx_{n}\to \gr_{nd(d-1), nd^{2}}$, with the image being the Springer fiber $Y_{n}$.

\begin{proof}[Proof of proposition \ref{stablequotient}]

Since the quotient $\xxs/T$ doesn't depend on the choice of $\xi$, we can suppose that $\sum_{i=1}^{d}\xi_{i}=0,\,\xi_{i}\in \bq$ is positive and small enough for $i=1,\cdots,d-1$.

Let $E_{i}=\kp^{-n}e_{i}/\kp^{n(d-1)}e_{i}$, let $V=\varrho_{n}(L)\in Y_{n}\subset\gr_{nd(d-1),nd^{2}}$. For $\tau\in \mathfrak{S}_{d}$, we have the equality

\begin{equation}\label{eg1}
\dim(V\cap (E_{\tau(1)}\oplus\cdots\oplus E_{\tau(i)}))=-\ind(L\cap(Fe_{\tau(1)}\oplus\cdots\oplus Fe_{\tau(i)}))+n(d-1)i.
\end{equation}

Let $x_{i}=\frac{n(d-1)-\xi_{i}}{nd(d-1)}$. The hypothesis on $\xi$ implies that $x_{i}\in \bq,\,x_{i}>0$ and $\sum_{i=1}^{d}x_{i}=1.$
 
Take $r_{i},s_{i}\in \bn$ such that
$$
x_{d}=\frac{1}{1+\sum_{i=1}^{d-1}s_{i}/r_{i}},\quad x_{i}=\frac{s_{i}/r_{i}}{1+\sum_{i=1}^{d-1}s_{i}/r_{i}},\quad i=1,\cdots,d-1.
$$

Let $S_{n}\subset T$ be the torus defined in (\ref{soustore}) for the above $r_{i},s_{i}$. The equality (\ref{eg1}) implies that

\begin{equation*}
\xi_{\tau(1)}+\cdots+\xi_{\tau(i)}<\ind(L\cap(Fe_{\tau(1)}\oplus\cdots\oplus Fe_{\tau(i)}))
\end{equation*}

if and only if 

\begin{equation}\label{ee2}
\dim(V)(x_{\tau(1)}+\cdots+x_{\tau(i)})>\dim(V\cap(E_{\tau(1)}\oplus\cdots\oplus E_{\tau(i)})).
\end{equation}

Combining the proposition \ref{stA} and the proposition \ref{st1}, we get
$$
L\in\xxs_{n}\Longleftrightarrow V\in Y_{n}^{st}.
$$ 

This implies that the subset $\xxs_{n}$ is an open subvariety of $\xx_{n}$ by the theory of Mumford. Moreover, since $\xi$ is supposed to be generic and $\xi_{i}$ are positive and small enough for $i=1,\cdots, d-1$, the ``$<$'' in the inequality (\ref{ee2}) is the same as ``$\leq$''. That is to say that $Y_{n}^{ss}=Y_{n}^{st}$ and so the quotient $Y_{n}^{ss}/\!/S_{n}=Y_{n}^{st}/S_{n}$ is a projective $k$-variety by the geometric invariant theory of Mumford. The following lemma shows that $\xxs_{n}/T=\xx^{\xi}_{n}/\!/S_{n}\cong Y_{n}^{ss}/\!/S_{n}$, since $T/Z_{G}$ acts freely on $\xx^{\xi}_{n}$. So $\xx_{n}^{\xi}/T$ is a projective $k$-variety.
\end{proof}

\begin{lem}

The morphism $S_{n}\to T/Z_{G}$ is an isogeny.

\end{lem}

\begin{proof}

It is equivalent to show that the induced morphism of character groups $X^{*}(T/Z_{G})\to X^{*}(S_{n})$ has non zero determinant. By a direct calculation, this determinant is
$$
\big(1+\sum_{i=1}^{d-1}s_{i}/r_{i}\big)\prod_{i=1}^{d-1}r_{i},
$$
which is non zero by the choice of $r_{i},\,s_{i}$.

\end{proof}

\subsection{The groups $\Sp_{2d}$ and $\SO_{2d}$}

Let $(k^{2d},\langle\,,\rangle)$ be the standard symplectic vector space over $k$ such that $\langle e_{i},e_{2d+1-i}\rangle=\delta_{i,j},\,i,j=1,\cdots,d$. Let $\Sp_{2d}$ be the symplectic group associated to it, let $T$ be the maximal torus of $\Sp_{2d}$ consisting of the diagonal matrices. Let $(F^{2d},\langle\,,\rangle)$ be the scalar extension of $(k^{2d},\langle\,,\rangle)$ to $F$. For a lattice $L$ in $F^{2d}$, let 
$$
L^{\vee}=\{x\in F^{2d}\mid \langle x,L\rangle \subset \co\}.
$$
The affine grassmannian associated to $\Sp_{2d}$ classifies the lattices $L$ in $F^{2d}$ such that $L=L^{\vee}$. Let 
$$
\xx_{n}=\{L\in \xx\mid \ep^{n}L_{0}\subset L\subset \ep^{-n}L_{0}\}.
$$
It is a $T$-invariant projective $k$-variety and we have $\xx=\lim_{n\to +\infty}\xx_{n}$. Let $\rho_{n}:\xx_{n}\to\gr_{2nd,4nd}$ be the injective $T$-equivariant morphism defined by 
$$
\rho_{n}(L)=L/\ep^{n}L_{0}\subset \ep^{-n}L_{0}/\ep^{n}L_{0}.
$$
Let $Y_{n}$ be its image, it is isomorphic to $\xx_{n}$. Let $\iota:\gr_{2nd,4nd}\to \bp^{N},\,N=\binom{4nd}{2nd}-1$, be the Pl\"ucker embedding. Let $\cl=(\iota\circ \rho_{n})^{*}\co_{\bp^{N}}(1)$, it is an ample $T$-equivariant line bundle on $Y_{n}$.

Let $\Gsp_{2d}$ be the reductive group over $k$ such that for any $k$-algebra $R$,
$$
\Gsp_{2d}(R)=\big\{g\in \gl_{2d}(R)\mid \langle gv, gv'\rangle =\lambda(g)\langle v, v'\rangle,\, \lambda(g)\in R^{\times},\,\forall\, v,\,v'\in R^{2d}\big\}.
$$
We have an exact sequence
$$
0\to \Sp_{2d} \to \Gsp_{2d}\xrightarrow{\lambda}\bbg_{m} \to 0,
$$
from which it follows that $\xx$ is the neutral connected component of $\xx^{\Gsp_{2d}}$. Let
$$
\widetilde{T}=\left\{\begin{bmatrix}tt_{1}&&&&&\\& \ddots&&&&\\ && tt_{d}&&&\\ &&&t_{d}^{-1}&&\\ &&&&\ddots &\\ &&&&& t_{1}^{-1}
\end{bmatrix};\; t,\,t_{i}\in k^{\times}
\right\}.
$$
It is a maximal torus of $\Gsp_{2d}$. Let $\bbg_{m}$ be the center of $\Gsp_{2d}$, then $\widetilde{T}/\bbg_{m}$ acts freely on $\xx$ and we have
$$
\xx^{\xi}/\widetilde{T}=\xx^{\xi}/T.
$$

Given a generic element $\xi=(\xi_{1},\cdots,\xi_{d},-\xi_{d},\cdots,-\xi_{1})\in \kt$ such that $\xi_{i}\in \bq$, we can find $r_{i},\,s_{i}\in \bz,\,i=1,\cdots,d$, such that
$$
\xi_{i}=nd\frac{s_{i}/r_{i}}{2+\sum_{i=1}^{d}s_{i}/r_{i}},\quad i=1,\cdots,d
$$
Consider the sub torus $S_{n}$ of $\widetilde{T}$ defined by
$$
S_{n}=\left\{\begin{bmatrix}t_{1}^{s_{1}}\cdots t_{d}^{s_{d}}t_{1}^{r_{1}}&&&&&\\& \ddots&&&&\\ && t_{1}^{s_{1}}\cdots t_{d}^{s_{d}}t_{d}^{r_{d}}&&&\\ &&&t_{d}^{-r_{d}}&&\\ &&&&\ddots &\\ &&&&& t_{1}^{-r_{1}}
\end{bmatrix};\; t_{i}\in k^{\times}
\right\}.
$$

\begin{lem}

The morphism $S_{n}\to \widetilde{T}/\bbg_{m}$ is an isogeny.

\end{lem}

\begin{proof}

As before, we need to calculate the determinant of the morphism of character groups $X^{*}(\widetilde{T}/\bbg_{m})\to X^{*}(S_{n})$, it is 
$$
\big(2+\sum_{i=1}^{d}s_{i}/r_{i} \big)\prod_{i=1}^{d}r_{i},
$$
which is non zero by the definition of $r_{i},\,s_{i}$.
\end{proof}

As we have done for $\gl_{d}$, we can calculate the $S_{n}$-(semi)-stable points $Y_{n}^{st}$ (resp. $Y_{n}^{ss}$) on $Y_{n}$ with respect to the polarization given by the line bundle $\cl$, and obtain the following comparison result, which implies the proposition \ref{mainquotient}.

\begin{prop}

Under the above setting, a lattice $L\in \xx_{n}$ is $\xi$-stable if and only if $\rho_{n}(L)\in Y_{n}^{ss}=Y_{n}^{st}$. As a consequence, $\xx_{n}^{\xi}$ is an open subvariety of $\xx_{n}$ and the quotient $\xx_{n}^{\xi}/T=\xx_{n}^{\xi}/(\widetilde{T}/\bbg_{m})\cong Y_{n}^{ss}/\!/S_{n}$ is a projective $k$-variety.   

\end{prop}

For the group $\SO_{2d}$, the strategy is totally the same, unless we need to use the standard quadratic space $(k^{2d},\langle\,, \rangle)$ over $k$ such that $\langle e_{i},e_{2d+1-i}\rangle=\delta_{i,j},\,i,j=1,\cdots,2d$, instead of the standard symplectic vector space.

\subsection{The group $\SO_{2d+1}$}

Let $(k^{2d+1},\langle\,,\rangle)$ be the standard quadratic space over $k$ such that $\langle e_{i},e_{2d+2-i}\rangle=\delta_{i,j},\,i,j=1,\cdots,2d+1$. Let $\SO_{2d+1}$ be the orthogonal group associated to it, and let $T$ be the maximal torus of $\SO_{2d+1}$ consisting of the diagonal matrices. Let $(F^{2d+1},\langle\,,\rangle)$ be the scalar extension of $(k^{2d+1},\langle\,,\rangle)$ to $F$. For a lattice $L$ in $F^{2d+1}$, let 
$$
L^{\vee}=\{x\in F^{2d+1}\mid \langle x,L\rangle \subset \co\}.
$$
The affine grassmannian associated to $\SO_{2d+1}$ classifies the lattices $L$ in $F^{2d+1}$ such that $L=L^{\vee}$. Let 
$$
\xx_{n}=\{L\in \xx\mid \ep^{n}L_{0}\subset L\subset \ep^{-n}L_{0}\}.
$$
It is a $T$-invariant projective $k$-variety and 
$$
\xx=\lim_{n\to +\infty}\xx_{n}.
$$

Let $\rho_{n}:\xx_{n}\to\gr_{n(2d+1),2n(2d+1)}$ be the injective $T$-equivariant morphism defined by 
$$
\rho_{n}(L)=L/\ep^{n}L_{0}\subset \ep^{-n}L_{0}/\ep^{n}L_{0}.
$$
Let $Y_{n}$ be its image, it is isomorphic to $\xx_{n}$. Let $\cl$ again be the $T$-equivariant line bundle on $Y_{n}$ induced by the Pl\"ucker embedding of $\gr_{n(2d+1),2n(2d+1)}$.

Let $\GO_{2d+1}$ be the reductive group over $k$ such that for any $k$-algebra $R$,
$$
\GO_{2d+1}(R)=\big\{g\in \gl_{2d+1}(R)\mid \langle gv, gv'\rangle =\lambda(g)\langle v, v'\rangle,\, \lambda(g)\in R^{\times},\,\forall\, v,\,v'\in R^{2d+1}\big\}.
$$
We have an exact sequence
$$
0\to \SO_{2d+1} \to \GO_{2d+1}\xrightarrow{\lambda}\bbg_{m} \to 0,
$$
from which it follows that $\xx$ is the neutral connected component of $\xx^{\GO_{2d+1}}$. Let
$$
\widetilde{T}=\left\{\begin{bmatrix}t^{2}t_{1}&&&&&&\\& \ddots&&&&&\\ && t^{2}t_{d}&&&&\\ &&&t&&&\\  &&&&t_{d}^{-1}&&\\ &&&&&\ddots &\\ &&&&&& t_{1}^{-1}
\end{bmatrix};\; t,\,t_{i}\in k^{\times}
\right\}.
$$
It is a maximal torus of $\GO_{2d+1}$. Let $\bbg_{m}$ be the center of $\GO_{2d+1}$, then $\widetilde{T}/\bbg_{m}$ acts freely on $\xx$ and we have
$$
\xx^{\xi}/\widetilde{T}=\xx^{\xi}/T.
$$

Given a generic element $\xi=(\xi_{1},\cdots,\xi_{d},0,-\xi_{d},\cdots,-\xi_{1})\in \kt$ such that $\xi_{i}\in \bq$, we can find $r_{i},\,s_{i}\in \bz,\,i=1,\cdots,d$, such that
$$
\xi_{i}=n\big(d+\frac{1}{2}\big)\frac{s_{i}/r_{i}}{1+\sum_{i=1}^{d}s_{i}/r_{i}},\quad i=1,\cdots,d
$$
Consider the sub torus $S_{n}$ of $\widetilde{T}$ defined by
$$
S_{n}=\left\{\begin{bmatrix}t_{1}^{2s_{1}}\cdots t_{d}^{2s_{d}}t_{1}^{r_{1}}&&&&&&\\& \ddots&&&&&\\ && t_{1}^{2s_{1}}\cdots t_{d}^{2s_{d}}t_{d}^{r_{d}}&&&&\\ &&& t_{1}^{s_{1}}\cdots t_{d}^{s_{d}}&&& \\ &&&&t_{d}^{-r_{d}}&&\\ &&&&&\ddots &\\ &&&&&& t_{1}^{-r_{1}}
\end{bmatrix};\; t_{i}\in k^{\times}
\right\}.
$$

\begin{lem}

The morphism $S_{n}\to \widetilde{T}/\bbg_{m}$ is an isogeny.
\end{lem}

\begin{proof}

As before, we need to calculate the determinant of the morphism of character groups $X^{*}(\widetilde{T}/\bbg_{m})\to X^{*}(S_{n})$, it is 
$$
\big(1+\sum_{i=1}^{d}s_{i}/r_{i} \big)\prod_{i=1}^{d}r_{i},
$$
which is non zero by the definition of $r_{i},\,s_{i}$.
\end{proof}

As before, we can calculate the $S_{n}$-(semi)-stable points $Y_{n}^{st}$ (resp. $Y_{n}^{ss}$) on $Y_{n}$ with respect to the polarization given by the line bundle $\cl$, and obtain the following comparison result. It implies the proposition \ref{mainquotient}.

\begin{prop}

Under the above setting, a lattice $L\in \xx_{n}$ is $\xi$-stable if and only if $\rho_{n}(L)\in Y_{n}^{ss}=Y_{n}^{st}$. As a consequence, $\xx_{n}^{\xi}$ is an open subvariety of $\xx_{n}$ and the quotient $\xx_{n}^{\xi}/T=\xx_{n}^{\xi}/(\widetilde{T}/\bbg_{m})\cong Y_{n}^{ss}/\!/S_{n}$ is a projective $k$-variety.   

\end{prop}

\section{Arthur-Kottwitz reduction}

Similarly to the Harder-Narasimhan reduction of non-semi-stable vector bundles on a smooth projective algebraic curve, we can stratify $\xx^{\text{n-}\xi}$ such that each strata is an infinite-dimensional homogeneous affine fibration on the various $\xx^{M,\xi}$ for $M\in \cl(T)$. We call this process the Arthur-Kottwitz reduction.

For $P\in \cf(T)$, let $P=MN$ be the standard Levi factorization. Let $\Phi_P(G,M)$ be the image of $\Phi(N,\,T)$ in $(\ka_M^G)^{*}$. For any point $a\in \ka_M^G$, we define a cone in $\ka_M^G$,
$$
D_P(a)=\left\{y\in\ka_M^G\,|\, \alpha(y-a)\geq 0,\,\forall \alpha\in \Phi_P(G,M)\right\}.
$$

\begin{defn}

For any geometric point $x\in \xx$, we define a semi-cylinder $C_P(x)$ in $\ka_T^G$ by
$$
C_P(x)=\pi^{M,-1}(\ec^{M}(f_{P}(x)))\cap \pi_{M}^{-1}(D_P(H_{P}(x))).
$$
\end{defn}

By definition, we get a partition 
$$
\ka_T^{G}=\ec(x)\cup\bigcup_{P\in \cf(T)}C_P(x).
$$ 
such that the interior of any two members doesn't intersect. The  figure \ref{polyred} gives an idea of this partition for $\gl_{3}$.

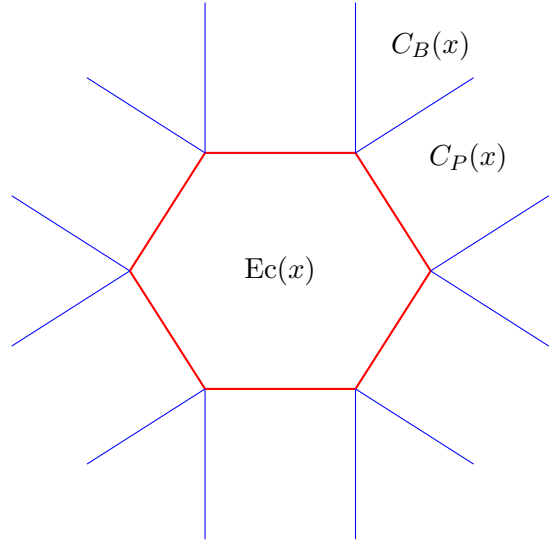
\begin{figure}[h]
\begin{center}

\begin{tikzpicture}

\draw[red, thick] (-1,1.57)--(1,1.57);
\draw[red, thick] (-1,1.57)--(-2,0);
\draw[red, thick] (-2,0)--(-1,-1.57);
\draw[red, thick] (-1,-1.57)--(1,-1.57);
\draw[red, thick] (1,-1.57)--(2,0);
\draw[red, thick] (2,0)--(1,1.57);

\draw[blue] (1,1.57)--(1,3.57);
\draw[blue] (1,1.57)--(2.57,2.57);
\draw[blue] (2,0)--(3.57,1);
\draw[blue] (2,0)--(3.57,-1);
\draw[blue] (1,-1.57)--(1,-3.57);
\draw[blue] (1,-1.57)--(2.57,-2.57);

\draw[blue] (-1,1.57)--(-1,3.57);
\draw[blue] (-1,1.57)--(-2.57,2.57);
\draw[blue] (-2,0)--(-3.57,1);
\draw[blue] (-2,0)--(-3.57,-1);
\draw[blue] (-1,-1.57)--(-1,-3.57);
\draw[blue] (-1,-1.57)--(-2.57,-2.57);

\draw (0,0) node{$\ec(x)$};
\draw (2, 3) node{$C_{B}(x)$};
\draw (2.5,1.5) node{$C_{P}(x)$};

\end{tikzpicture}
\caption{$\ec(x)$ and $C_{P}(x)$ for $\gl_{3}$.}
\label{polyred}
\end{center}
\end{figure}

So for any $x\notin \xxs$, there exists a unique parabolic subgroup  $P\in \cf(T)$ such that $\xi\in C_P(x)$ since $\xi$ is generic. In this case, $f_P(x)\in \xx^M$ is $\xi^{M}$-stable, where $\xi^{M}=\pi^{M}(\xi)\in \ka_T^M$. Let
$$
S_P=\{x\in \xx| \,\xi\in C_{P}(x)\},
$$
we have the decomposition of the affine grassmannian
\begin{equation}\label{decomp1}
\xx=\xxs\sqcup \bigsqcup_{P\in \cf(T),\,P\neq G}S_P.
\end{equation}

For $P\in \cp(M)$, let $P^{-}$ be the parabolic subgroup opposite to $P$ with respect to $M$. Let $\Lambda_{M,P}=D_{P^{-}}(\xi_{M})\cap \Lambda_{M}$, where $\xi_{M}=\pi_{M}(\xi)\in \ka_{M}^{G}$. We have the disjoint partition
$$
\Lambda_M=\bigsqcup_{P\in \cp(M)}\Lambda_{M,P}.
$$
For $\lambda\in \Lambda_{M}$, let
$
\xx^{M,\lambda,\xi^M}=\xx^{M,\xi^M}\cap \xx^{M,\lambda}.
$
By definition, we have $$f_{P}(S_{P})=\bigsqcup_{\lambda\in \Lambda_{M,P}}\xx^{M,\lambda, \xi_{M}}.$$ Recall that $\xx^{M}$ is embedded naturally in $\xx$ via the inclusion $M(F)\to G(F)$. Hence, $\xx^{M,\lambda,\xi^{M}}$ is embedded naturally in $S_{P}$.

\begin{prop}\label{red2}

The group action of $N(F)$ on $\xx$ by left translation preserves the subset $S_{P}$, from which we get the decomposition  
$$
S_{P}=\bigsqcup_{\lambda\in \Lambda_{M,P}} N(F)\xx^{M,\lambda,\xi^{M}}.
$$
Under the retraction $f_{P}$, each orbit $N(F)\xx^{M,\lambda,\xi^{M}}$ becomes an infinite dimensional homogeneous affine fibration on $\xx^{M,\lambda,\xi^{M}}$. Furthermore, each orbit is locally closed in $\xx$.

\end{prop}

\begin{proof}

Since $f_P(ux)=f_P(x)$ for any $x\in  S_P$ and any $u \in N(F)$, the group $N(F)$ preserves the subset $S_{P}$. By Iwasawa decomposition, each $N(F)$-orbit on $\xx$ passes through a unique point on $\xx^{M}$. When restricted to $S_{P}$, this point lies in $\bigsqcup_{\lambda\in \Lambda_{M,P}} \xx^{M,\lambda,\xi^{M}}$, from which follows the decomposition. The remaining assertions follow from proposition \ref{fibrationp} and the fact that $\xx^{M,\lambda,\xi^{M}}$ is open in $\xx^{M,\lambda}$.

\end{proof}

We summarise the above discussion into the following theorem.

\begin{thm}\label{red1}

The affine grassmannian can be decomposed into locally closed sub-ind-$k$-schemes
$$
\xx=\xx^{\xi}\sqcup \bigsqcup_{\stackrel{P\in \cf(T)}{P\neq G}}\bigsqcup_{\lambda\in \Lambda_{M_{P},P}}N_{P}(F)\xx^{M_{P},\lambda,\xi^{M_{P}}},
$$
where $P=M_{P}N_{P}$ is the standard Levi decomposition of $P$. Each strata $N_{P}(F)\xx^{M_{P},\lambda,\xi^{M_{P}}}$ is an infinite dimensional homogeneous affine fibration over $\xx^{M_{P},\lambda,\xi^{M_{P}}}$.

\end{thm}

We call the above decomposition the \emph{Arthur-Kottwitz reduction}. We want to point out that each strata $N_{P}(F)\xx^{M_{P},\lambda,\xi^{M_{P}}}$ is not only of infinite dimension, but also of infinite codimension in the affine grassmannian $\xx$.

\section{Poincaré series of $\xx^{\xi}/T$ for $\SL_{d}$}

Let $k=\overline{\mathbf{F}}_{p}$, let $l$ be a prime number different from $p$. Let $G=\SL_{d}$, let $T$ be the maximal torus of $G$ consisting of the diagonal matrices, let $B$ be the Borel subgroup of $G$ consisting of the upper triangular matrices. Let $X_{*}^{+}(T)$ be the cone of dominant cocharacters $\mu$ of $T$ with respect to $B$.

\subsection{Poincaré series of the affine grassmannian}

Let $V$ be a separated scheme of finite type over $k$, we use the notation:
$$
H_{i}(V)=(H_{\mathrm{et}}^{i}(V,\overline{\bq}_{l}))^{*},\quad H_{i,c}(V)=(H^{i}_{\mathrm{et},c}(V,\overline{\bq}_{l}))^{*}. 
$$
Its Poincaré polynomial is defined to be
$$
P_{V}(t)=\sum_{i=0}^{2\dim(V)}\dim(H_i(V))t^{i}.
$$
For an ind-$k$-scheme $\cvv=\injlim V_{n}$, we define
$$
H_{i}(\cvv)=\lim_{n\to\infty}H_{i}(V_{n}),\quad H_{i,c}(\cvv)=\lim_{n\to\infty}H_{i,c}(V_{n}).
$$
If $\dim(H_{i}(\cvv))<+\infty$ for all $i\in \bn$, we define the Poincaré series of $\cvv$ to be
$$
P_{\cvv}(t)=\sum_{i=0}^{\infty}\dim(H_i(\cvv))t^{i}.
$$

Let $I$ be the standard Iwahori subgroup of $G(F)$, i.e. it is the inverse image of $B$ under the reduction morphism $G(\co)\to G(k)$. Recall that we have the Bruhat-Tits decomposition of the affine grassmannian in affine spaces,  
$$
\xx=\bigsqcup_{\lambda\in X_{*}(T)}I\ep^{\lambda}K/K,
$$
and $I\ep^{\lambda}K/K\subset \overline{I\ep^{\mu}K/K}$ if and only if $\lambda\prec \mu$, which means the following: For $\lambda,\mu\in X_{*}^{+}(T)$, then $\lambda\prec \mu$ if and only if
\begin{eqnarray*}
\lambda_{1}+\cdots+\lambda_{i}&\leq& \mu_{1}+\cdots+\mu_{i},\quad i=1,\cdots,d.
\end{eqnarray*}
Then we impose $W\lambda\prec W\mu$. For all $g,g'\in W/W_{\lambda}$, where $W_{\lambda}$ is the stabilizer of $\lambda$, we impose that $g\lambda\prec g'\lambda$ if and only if $g'\prec_{B}g$, where $\prec_{B}$ is the order on $W/W_{\lambda}$ induced by the Bruhat-Tits order on $W$ with respect to $B$.

For $\mu\in X_{*}(T)$, let $\sch(\mu)=\overline{I\ep^{\mu} K/K}$. Then we have
$$
\sch(\mu)=\bigcup_{\substack{\lambda\in X_{*}(T)\\ \lambda\prec \mu}}I\ep^{\lambda}K/K.
$$

Let $\nu=((d-1)n, -n,\cdots, -n)\in X_{*}^{+}(T)$, it is easy to see that $\xx_{n}=\overline{I\ep^{\nu}K/K}$. Using the Bruhat-Tits decomposition, one can prove that

\begin{prop}[Bott]

The Poincaré series of the affine grassmannian is 
$$
P_{\xx}(t)=\prod_{i=1}^{d-1}(1-t^{2i})^{-1}.
$$
\end{prop}

The reader can find a topological proof in \cite{bott}, and a combinatorial proof in \cite{iwahori}.

\subsection{$\xx^{\xi}/T$ is homologically smooth}

Let $V$ be an irreducible separated scheme of finite type over $k$, recall that $V$ is called \emph{homologically smooth} if it satisfies $\mathrm{IC_{V}^{\bullet}}=\overline{\bq}_{l,V}$, where $\mathrm{IC_{V}^{\bullet}}=j_{!*}\overline{\bq}_{l, U}$ for any dense open smooth subvariety $U\xrightarrow{j} V$.

\begin{lem}\label{partiepure}

For $n\in \bn$, the algebraic variety $\xx_{n}$ is homologically smooth.

\end{lem}

\begin{proof}

Let $\mu\in X_{*}(T), \mu\prec \nu$. Let $M_{\nu}$ be the irreducible representation of $\pgl_{d}$ with highest weight $\nu$, it can also be seen as an irreducible representation of $\gl_{d}$. Let $n_{\mu}$ be the dimension of the weight space of weight $\mu$ in $M_{\nu}$. It can be calculated to be $1$ by the traditional Young tableaux.

By the geometric Satake isomorphism \cite{mv}, the dimension of the stalk of $\mathrm{IC_{\xx_{n}}^{\bullet}}$ at $\ep^{\mu}$ is the same as $n_{\mu}=1$, since $\pgl_{d}$ is the Langlands dual group of $\SL_{d}$. So $\mathrm{IC_{\xx_{n}}^{\bullet}}=\overline{\bq}_{l,\xx_{n}}$, as claimed.

\end{proof}

\begin{cor}\label{stablepure}

For any $n\in \bn$, the algebraic variety $\xx_{n}^{\xi}/T$ is homologically smooth. In particular, it satisfies the Poincaré duality.
\end{cor}

\begin{proof}

Since $\xxs_{n}$ is open in $\xx_{n}$, $\xxs_{n}$ is homologically smooth. It is a $T/\bbg_{m}$-torsor over $\xxs_{n}/T$ since the action of $T/\bbg_{m}$ is free. In particular, the natural projection $\xxs_{n}\to \xxs_{n}/T$ is a surjective smooth morphism. According to \cite{gr}, corollary 17.16.3, locally it admits an étale section. So the torsor is locally trivial for the étale topology. This implies that $\xxs_{n}/T$ is homologically smooth since $\xxs_{n}$ is.
\end{proof}

\subsection{Calculation of the Poincaré series}

\begin{lem}\label{dim1}

Let $\xi\in \ka_{T}^{G}$ be such that $0<\alpha(\xi)<1,\,\forall \alpha\in \Phi_{B}(G,T)$. We have $\dim(\xx_{n})=nd(d-1)$, and the closed sub-variety $\xx_{n}\backslash \xxs_{n}$ have dimension at most $n(d-1)^{2}$, i.e. its codimension is at least $(d-1)n$ in $\xx_{n}$.
\end{lem}

\begin{proof}

Since $\xx_{n}=\overline{I\ep^{\nu} K/K}$, we have $\dim(\xx_{n})=\dim(I\ep^{\nu} K/K)=nd(d-1)$. 
The dimension of the closed sub-variety $\xx_{n}\backslash \xxs_{n}$ is
$$
\max\{\dim(S_{P}\cap \xx_{n}),\,P\in \cf(T),\,P\neq G\}.
$$

An easy induction reduces the situation to the case where $P$ is a standard maximal parabolic subgroup, here standard means that $B\subset P$. Suppose that $P$ is of type $(r,d-r)$, $1 \leq r \leq d-1$, i.e. its Levi factor is $M=\mathrm{S}(\aut(k^{r})\times \aut(k^{d-r}))$.

We identify $\Lambda_{M}$ with $\{(\lambda,-\lambda)\mid \lambda\in \bz\}$. For $(\lambda,-\lambda)\in \Lambda_{M,P}$, we have $\lambda\leq 0$. Let $X^{\lambda}=H_{P}^{-1}((\lambda,-\lambda))\cap \xx_{n}$, then the $S_{P}^{\lambda}:=S_{P}\cap X^{\lambda}$ are the connected components of $S_{P}\cap \xx_{n}$. So it is enough to bound the dimension of $X^{\lambda}$.

\begin{lem}

We have the affine paving
$$
\xx_{n}=\bigcup_{\ep^{\mu}\in \xx_{n}^{T}}\xx_{n}\cap B\ep^{\mu}K/K,
$$
where
$$
\xx_{n}\cap B\ep^{\mu}K/K=\begin{bmatrix}1&\cdots&\kp^{a_{i,j}}\\
&\ddots&\vdots\\
&&1
\end{bmatrix}\ep^{\mu}K/K
$$
with $a_{i,j}=-n-\mu_{j}$, is isomorphic to an affine space of dimension 
$$
\sum_{i=1}^{d-1}(d-i)(n+\mu_{i}).
$$

\end{lem}

For the proof, the reader can refer to \cite{chen}, corollaire 2.3. For $\ep^{\mu}\in \xx_{n}^{T}$, let $C(\mu)=\xx_{n}\cap B\ep^{\mu}K/K$. Since $B\subset P$, we have 
$$
X^{\lambda}=\bigsqcup_{\ep^{\mu}\in (X^{\lambda})^{T}}C(\mu).
$$

So the question is to bound the dimension of $C(\mu)$ under the condition that $\mu_{i}\geq -n$ and that
$$
\sum_{i=1}^{r}\mu_{i}=\lambda\leq 0,\quad \sum_{i=r+1}^{d}\mu_{i}=-\lambda\geq 0.
$$

It takes the maximal value when

\begin{eqnarray*}
&&\mu_{1}=\lambda+(r-1)n, \, \mu_{2}=\cdots=\mu_{r}=-n; \\
&&\mu_{r+1}=-\lambda+(d-r-1)n,\,\mu_{r+2}=\cdots=\mu_{d}=-n.
\end{eqnarray*}

So we have
 
\begin{eqnarray*}
\dim(X^{\lambda})&\leq& (d-1)(\lambda+rn)+(d-(r+1))((d-r)n-\lambda)\\
&\leq& (d-1)^{2}n,
\end{eqnarray*}

and then

$$
\codim (\xx_{n}\backslash\xxs_{n})\geq nd(d-1)-(d-1)^{2}n=(d-1)n.
$$
\end{proof}

\begin{rem}

We can also use the dimension formula in theorem 3.2 of \cite{mv} to obtain the same estimation.

\end{rem}

For $n\in \bn$, let $\tau_{n}$ be the truncation operator on $k[\![t]\!]$ defined by 
$$
\tau_{n}\left(\sum_{i=i_{0}}^{+\infty}a_{i}t^{i}\right)=\sum_{i=i_{0}}^{n}a_{i}t^{i}.
$$

\begin{thm}

The Poincaré series of $\xxs/T$ is
\begin{eqnarray*}
P_{\xxs/T}(t)=\frac{1}{(1-t^{2})^{d-1}}\prod_{i=1}^{d-1}(1-t^{2i})^{-1}.
\end{eqnarray*}

Furthermore, we have
$$
H_{2i+1}(\xx^{\xi}/T)=0,
$$
and the Frobenius acts on $H_{2i}(\xx^{\xi}/T)$ by $q^{-i}$, $\forall i\geq 0$.
\end{thm}

\begin{proof}

We have the exact sequence
\begin{eqnarray*}
\cdots \to H_{i}(\xx_{n}\backslash \xx_{n}^{\xi})\to H_{i}(\xx_{n})\to H_{i,c}(\xx_{n}^{\xi}) \to H_{i-1}(\xx_{n}\backslash \xx_{n}^{\xi})\to\cdots
\end{eqnarray*}

By lemme \ref{dim1}, 
$$
\dim(\xx_{n}\backslash \xx_{n}^{\xi})\leq n(d-1)^{2},
$$ 

so  

$$
H_{i}(\xx_{n}\backslash \xx_{n}^{\xi})=0,\quad i\geq 2n(d-1)^{2}+1,
$$
and
$$
H_{i,c}(\xx_{n}^{\xi})=H_{i}(\xx_{n}),\quad i\geq 2n(d-1)^{2}+2.
$$

Since $\xx_{n}$ and $\xx_{n}^{\xi}$ are homologically smooth, they satisfy the Poincaré duality, which implies that

\begin{equation}\label{smalldim}
H_{i}(\xx_{n}^{\xi})=H_{i}(\xx_{n}),\quad 0\leq i\leq 2n(d-1)-2,
\end{equation}
since $\dim(\xx_{n})=\dim(\xx_{n}^{\xi})=nd(d-1)$.

Because $T/\bbg_{m}$ acts freely on $\xx_{n}^{\xi}$, we have

\begin{eqnarray}\label{kunneth}
H_{i}(\xx^{\xi}_{n}/T)&=&H_{i,T/\bbg_{m}}(\xx^{\xi}_{n}) \nonumber \\
&=&\bigoplus_{i_{1}+i_{2}=i}H_{i_{1}}(\xx_{n}^{\xi})\otimes H_{i_{2}}(B(T/\bbg_{m})),
\end{eqnarray}
where $B(T/\bbg_{m})$ is the classifying space of $T/\bbg_{m}$-torsors.

Combining the equalities (\ref{smalldim}) and (\ref{kunneth}), we get

\begin{equation}\label{compare}
\tau_{2(d-1)n-2}[P_{\xxs_{n}/T}(t)]=\tau_{2(d-1)n-2}[(1-t^{2})^{1-d}P_{\xx_{n}}(t)],
\end{equation}
and
$$
H_{2i+1}(\xx^{\xi}_{n}/T)=0,\quad 0\leq i\leq (d-1)n-2,
$$
and the Frobenius acts on $H_{2i}(\xx^{\xi}_{n}/T)$ by $q^{-i}$, $0\leq i\leq (d-1)n-1$, because it acts as such on $H_{*}(\xx_{n})$ and $H_{*}(B(T/\bbg_{m}))$.

Since

$$
\lim_{n\to +\infty}\tau_{2(d-1)n}(P_{\xx_{n}}(t))=P_{\xx}(t)=\prod_{i=1}^{d-1}(1-t^{2i})^{-1},
$$ 
we can take limits of the two sides of (\ref{compare}), and get:
$$
P_{\xxs/T}(t)=\frac{1}{(1-t^{2})^{d-1}}\prod_{i=1}^{d-1}(1-t^{2i})^{-1}.$$

This implies that we can also take limits of the two sides of (\ref{smalldim}) and (\ref{kunneth}), and obtain  

$$
H_{i}(\xx^{\xi}/T)=\bigoplus_{i_{1}+i_{2}=i}H_{i_{1}}(\xx)\otimes H_{i_{2}}(B(T/\bbg_{m})),
$$
and the second part of the theorem follows.
\end{proof}

\begin{rem}
The rotation torus $\bbg_{m}$ acts on the quotient $\xxs/T$. It gives an affine paving of $\xxs/T$ for $\SL_{2}$. But for $\SL_{d},\,d\geq 3$, the fixed points $(\xxs/T)^{\bbg_{m}}$ are not discrete.
\end{rem}

\begin{rem}

For other classical groups $G$, we expect also the relation 
$$
H_{*}(\xx^{\xi}/T)=H_{*}(BT)\otimes H_{*}(\xx),
$$
hence
$$
P_{\xx^{\xi}/T}(t)=\frac{1}{(1-t^{2})^{\mathrm{rk}(G)}}P_{\xx}(t),
$$
where $\mathrm{rk}(G)$ is the semisimple rank of $G$. In fact, heuristically, we should think of the affine grassmannian as an ``infinite dimensional smooth projective algebraic variety'', since it is a homogeneous space and it is an injective limit of projective algebraic varieties. So there should be some way to extend our proof  to the other groups.

\end{rem}


\end{document}